\documentclass[11pt]{article}
\usepackage{mathrsfs}
\usepackage{amsmath}
\usepackage[colorlinks=false, pdfborder={0 0 0}]{hyperref}
\hypersetup{
    colorlinks=true,
    linkcolor=red
}
\usepackage{cleveref} 
\usepackage{amsfonts}
\usepackage{amssymb, cite}
\usepackage{color}
\usepackage{graphicx}
\usepackage{subcaption}
\usepackage{array}
\usepackage{booktabs}
\usepackage{diagbox}
\usepackage[T1]{fontenc}
\usepackage{tabularx}
\usepackage{amsthm}

\numberwithin{equation}{section}

\newtheorem{theorem}{Theorem}[section]
\newtheorem{lemma}{Lemma}[section]
\newtheorem{proposition}{Proposition}[section]
\newtheorem{remark}{Remark}[section]
\newtheorem{definition}{Definition}[section]

\newcommand\be{\begin{equation}}
\newcommand\ee{\end{equation}}
\newcommand\ber{\begin{eqnarray}}
\newcommand\eer{\end{eqnarray}}
\newcommand\berr{\begin{eqnarray*}}
\newcommand\eerr{\end{eqnarray*}}
\newcommand\bea{\begin{eqnarray}}
\newcommand\eea{\end{eqnarray}}

\title{ \bf\Large Minimizers and Weak Solutions for Singular Born--Infeld Type Functionals}
\author{
  {\bf Tengyang Liu}$^{a}$\thanks{Email: tyl619@henu.edu.cn; ORCiD: \href{https://orcid.org/0009-0009-1598-6638}{0009-0009-1598-6638}}
  {\bf and  Ruifeng Zhang}$^{a}$\thanks{Email: zrf615@henu.edu.cn; ORCiD: \href{https://orcid.org/0000-0002-4378-7427}{0000-0002-4378-7427}}\\[2mm]
  {\it\small $^a$College of Mathematics and Statistics, Henan University} \\[2mm]
  {\it\small P.R.China}
}
\date{}

\begin{document}

\maketitle

\begin{abstract}
In this paper, we investigate the relation between minimizers and weak solutions for a class of singular functionals arising from Born--Infeld type theories $\mathcal{L}(s)$. In the setting of an electrostatic field (hence $s=\frac{1}{2}|\nabla\phi|^2$), $\mathcal{L}(s)$ satisfies $\lim_{s\to(1/2)^-}\mathcal{L}(s)=+\infty$, which naturally enforces the finite gradient bound $|\nabla\phi|\le 1$, also called the truncation threshold. For a prescribed extended charge density $\rho$, we consider the relation between the weak solution of the system
\begin{equation}
\begin{cases}
-{\rm div}\left(b\left(\frac12|\nabla\phi|^2\right)\nabla\phi\right)=\rho, & \text{in } \mathbb{R}^N, \\
b(s)=\mathcal{L}'(s),\quad \lim_{s\to\frac12^-} b(s)=+\infty, \\
\lim_{|x|\to\infty}\phi(x)=0
\end{cases}
\notag
\end{equation}
and the minimizer $\phi_0$ of the singular functional. We propose a novel monotonic approximation method to handle the intrinsic singularities of $\mathcal{L}(s)$. We prove that the gradient of the minimizer never touches the singular boundary $|\nabla\phi|^2=1$; this structural result yields the key integrability property $b(\frac12|\nabla\phi_0|^2)|\nabla\phi_0|^2\in L^1(\mathbb{R}^N)$, the existence and uniqueness of the minimizer, and the corresponding variational inequality. Under the additional assumption that $\rho$ is radially distributed, we show that the minimizer is the unique weak solution. Furthermore, we establish the $C^1$ and $C^2$ regularity of the minimizer under suitable integrability conditions on $\rho$, and provide a uniform estimate for the strict spacelikeness condition $|\nabla\phi_0|\le 1-\epsilon$, where the parameter $\epsilon>0$ is explicitly characterized in terms of the spatial dimension, the spatial region, and $\rho$. Our results extend the classical Born--Infeld theory to a general class of singular Born--Infeld type theories, thereby providing a unified framework for the variational analysis and regularity of such singular functionals and systems.
\\

\end{abstract}

\section{Introduction}\label{s1}

Almost a century ago, Born and Infeld \cite{Born1933a,Born1934a,Born1933b,Born1934b} proposed a nonlinear modification of classical electrodynamics, aiming to resolve the long-standing problem of the infinite self-energy of point charges in Maxwell's theory. In their theory, the electromagnetic field satisfies nonlinear equations. In the electrostatic setting, the Lagrangian density takes the form
\begin{align}
\mathcal{L}_{\mathrm{Born-Infeld}}=1-\sqrt{1-2s},\quad s=\frac12|\nabla\phi|^2,\notag
\end{align}
where $\phi$ denotes the electrostatic potential. For weak fields ($|\nabla\phi|\ll1$), this Lagrangian reduces to the classical Maxwell form $\frac12|\nabla \phi|^2$; for strong fields, it naturally imposes the gradient constraint $|\nabla \phi|^2\le1$. The corresponding Euler--Lagrange equation is
\begin{align}
-{\rm div}\left(\frac{\nabla \phi}{\sqrt{1-|\nabla \phi|^2}}\right)=\rho\quad\text{in }\mathbb{R}^N,\label{1.1}
\end{align}
where $\rho$ is the charge density. For the case of a point charge of $\rho$, the solution has finite energy, thereby resolving the infinite self-energy problem of the Maxwell theory.

The electrostatic Born--Infeld energy functional is defined as
\begin{align}
I_{\rm Born-Infeld}(\phi)=\int_{\mathbb{R}^N}\bigl(1-\sqrt{1-|\nabla \phi|^2}\bigr)\,{\rm d}x-\langle\rho,\phi\rangle,
\label{1.2}
\end{align}
formally, critical points of \eqref{1.2} satisfy equation \eqref{1.1}. Although this functional is well-defined only under the condition $|\nabla \phi|^2\le 1$ almost everywhere, \eqref{1.1} exhibits singular behavior as $|\nabla \phi|^2\to 1^-$; consequently, it is not a $C^1$ functional on any classical Sobolev space. This lack of smoothness necessitates the use of convex analysis, subdifferentials, and non-smooth critical point theory \cite{Haarala2020,Szulkin1986,Ekeland1999}. Moreover, the natural energy space embeds non-compactly into $L^p(\mathbb{R}^N)$, and the gradient constraint $|\nabla \phi|^2\le 1$ is not preserved under weak convergence. Hence, even the existence of minimizers requires careful analysis.

Depending on the nature of the right-hand side of \eqref{1.1}, the contemporary investigations can be divided into several directions. When $\rho$ is a prescribed charge distribution, the existence and uniqueness of a minimizer follow from convexity \cite{Bonheure2016}; the main questions concern whether this minimizer satisfies the Euler--Lagrange equation and what its regularity properties are. Bonheure, d'Avenia and Pomponio \cite{Bonheure2016} studied the electrostatic Born--Infeld equation with extended charges. They proved that when $\rho$ is radially distributed or $\rho\in L^\infty_{\mathrm{loc}}$, the minimizer is also a weak solution, and they introduced a truncation technique that approximates the classical Born--Infeld operator by a sequence of $p$-Laplacian type operators. Bonheure and Iacopetti \cite{Bonheure2019a} established $W^{2,2}_{\mathrm{loc}}(\mathbb{R}^N)$ regularity of the minimizer under the assumptions $\rho\in L^q\cap L^m$ with $q>2N$ and $m\in[1,2_*]$, and under a smallness condition they proved that the minimizer is strictly spacelike and belongs to $C^{1,\alpha}_{\mathrm{loc}}$. Haarala \cite{Haarala2020} extended the regularity result to the case $\rho\in L^p$ with $p>N$, removing the smallness assumption and showing that the minimizer satisfies $|\nabla \phi|\le1-\theta$ ($\theta>0$) and is of class $C^{1,\alpha}$. Bonheure, Colasuonno and F\"oldes \cite{Bonheure2019b} treated the case of superpositions of point charges $\rho=\sum a_i\delta_{x_i}$, proving that away from the charges the minimizer is a weak solution, and under suitably small charges it is smooth and strictly spacelike except at the points themselves.

When the right-hand side of \eqref{1.1} depends on the unknown function itself, the equation becomes the nonlinear field equation $-\operatorname{div}\bigl(\nabla \phi/\sqrt{1-|\nabla \phi|^2}\bigr)=g(\phi)$. Early works relied on ODE methods (e.g., Azzollini \cite{Azzollini2014,Azzollini2016}) or truncation techniques (e.g., Bonheure, Derlet and De Coster \cite{Bonheure2012}) to obtain radial solutions for power-type nonlinearities $g(\phi)=|\phi|^{p-2}\phi$ with $p>2^*$. Mederski and Pomponio \cite{Mederski2023} extended the study to more general nonlinearities under Berestycki--Lions type conditions and proved the existence of radial solutions with finite energy. Bieganowski, Ikoma and Mederski \cite{Bieganowski2025} developed a direct variational method on the Poho\v{z}aev manifold, established the existence of ground state solutions, and provided a new characterization of the optimal constant in the Sobolev-type inequality. In recent years, normalized solutions have also attracted attention; Baldelli, Mederski and Pomponio \cite{Baldelli2025} systematically studied normalized solutions for a large class of quasilinear operators including the classical Born--Infeld operator, obtaining ground states in the $L^2$-subcritical case via a truncation method and solutions in the $L^2$-supercritical case via a mountain-pass approach. A significant breakthrough for the classical Born--Infeld operator was recently obtained by Bonheure and Iacopetti \cite{Bonheure2023}, who proved a sharp gradient estimate for the prescribed mean curvature equation in Lorentz--Minkowski space. As a consequence, they showed that for \(\rho \in L^q(\mathbb{R}^N)\cap L^m(\mathbb{R}^N)\) with \(q>N\) and \(m\in[1,2_*]\), the unique minimizer of the Born--Infeld energy is strictly spacelike and belongs to \(W_{\rm loc}^{2,q}(\mathbb{R}^N)\cap C^{1,\alpha}_{\rm loc}(\mathbb{R}^N)\) with \(\alpha=1-N/q\), removing the smallness assumption previously required in \cite{Bonheure2012} and achieving the optimal regularity threshold.

As can be seen, most of the aforementioned works focus on the classical Born--Infeld operator ${\rm div}(\nabla \phi/\sqrt{1-|\nabla \phi|^2})$. However, the classical Born--Infeld theory itself belongs to a broader conceptual framework: \emph{Born--Infeld type theories} $\mathcal{L}(s)$, which refer to a class of nonlinear electrodynamics models, where $s=\frac12|\nabla \phi|^2$ is the Maxwell action (as defined in this paper). The most notable signature of $\mathcal{L}(s)$ is that it satisfies the following normalization conditions:
\begin{align}
\mathcal{L}(0)=0,\quad \mathcal{L}'(0)=1,\label{1.3}
\end{align}
which guarantee that in the weak-field limit, $s\to 0$, the classical Maxwell theory $\mathcal{L}_{\rm Maxwell}(s)=s$ is approximated. Recently, the interdisciplinary research between Born--Infeld type theory and PDEs has preliminarily entered the horizon of researchers. For example, the exponential model $\mathcal{L}_{\rm exp}(s)=\frac{1}{\beta}(\mathrm{e}^{\beta s}-1)$\cite{Hendi2012,Hendi2013}, where $\beta$ is a positive physical parameter and it also enables any Born--Infeld type theory to approximate the Maxwell theory when $\beta\to0$. Correspondingly, Dai and Zhang \cite{Dai2023a,Dai2023b} mathematically studied the existence of ground state solutions for $\mathcal{L}_{\rm exp}(s)$, using variational methods and dynamical shooting methods respectively, and established the existence of nontrivial nonnegative solutions on bounded domains and on the whole space $\mathbb{R}^N$. These conclusions indicate that considering the coupling problem of Born--Infeld type theories and PDEs is a novel direction. In particular, we have noticed a fact in \cite{Yang2022} that some Born--Infeld type theories (such as $\mathcal{L}_{\rm exp}(s)$) do not possess the natural physical {\em truncation threshold} on $s$ that the classical Born--Infeld theory does. This enables Born--Infeld type theories to be distinguished according to whether they possess a physical truncation threshold or not. Now we introduce a representative model with this characteristic: the rational-function model\cite{Kruglov2015,Kruglov20152}, whose Lagrangian is given by
\begin{align}
    \mathcal{L}_{\rm rational}(s)=\frac s{1-2s},\quad s\le\frac12,\quad b_{\rm rational}(s)=\mathcal{L}'_{\rm rational}(s)=\frac{1}{(1-2s)^2}.\label{1.4}
\end{align}
This nonlinear model originates from the theory of vacuum birefringence and can generate a singular functional
\begin{align}
    I_{\rm rational}=\frac12\int_{\mathbb{R}^N}\frac{ |\nabla\phi|^2}{1- |\nabla\phi|^2}\,{\rm{d}}x-\langle\rho,\phi\rangle,\label{rational functional}
\end{align}
and the electrostatic system also has singularity: 
\begin{equation}
\begin{cases}
-{\rm div}\left(\dfrac{\nabla\phi}{(1-|\nabla\phi|^2)^2}\right)=\rho, & \text{in } \mathbb{R}^N, \\
\lim_{|x|\to\infty}\phi(x)=0.
\end{cases}
\label{rational system}
\end{equation}
We will conduct a detailed study of \eqref{1.4} in Section \ref{s3}. However, the Born--Infeld type theories with truncation thresholds are diverse, and they all have unique nonlinear structures, which means that each Born--Infeld type theory with a truncation threshold will produce a singular functional that is not $C^1$, given by
\begin{align}
    I=\int_{\mathbb{R}^N} \mathcal{L}\left(s\right)\,{\rm{d}}x-\langle\rho,\phi\rangle:=J(\phi)-\langle\rho,\phi\rangle,\quad \lim_{s\to\frac12^-}\mathcal{L}(s)=+\infty,\quad s=\frac{1}{2}|\nabla\phi|^2,\label{1.5}
\end{align}
and the corresponding system is 
\begin{equation}
\begin{cases}
-{\rm div}\left(b\left(\frac12|\nabla\phi|^2\right)\nabla\phi\right)=\rho, & \text{in } \mathbb{R}^N, \\
b(s)=\mathcal{L}'(s),\quad \lim_{s\to\frac12^-} b(s)=+\infty, \\
\lim_{|x|\to\infty}\phi(x)=0.
\end{cases}
\label{1.6}
\end{equation}
Therefore, whether minimizers of such functionals are in fact weak solutions of \eqref{1.6} has become a difficult but fascinating question. Actually, for general and singular Born--Infeld type operators of the form ${\rm div}(b(\frac12|\nabla \phi|^2)\nabla \phi)$, Mederski and Pomponio \cite{Mederski2023} proved the existence of radial solutions in the context of nonlinear field equations $-\operatorname{div}(b(|\nabla \phi|^2)\nabla \phi)=g(\phi)$. However, for the case of a prescribed charge distribution $\rho$ (i.e., the equation $-\operatorname{div}(b(|\nabla \phi|^2)\nabla \phi)=\rho$), the relation between the minimizer and the weak solution within this general operator framework has not yet been systematically established. 

Motivated by the above, in this paper, from a mathematical perspective, we focus on and must verify  whether the minimizers of such singularity functionals will appear on the singularity boundary. By working on the following functional space
\begin{align}
    \mathcal{A}=D^{1,2}(\mathbb{R}^N)\cap\left\{\phi\in C^{0,1}(\mathbb{R}^N)\,\big|\,|\nabla\phi|^2\le1\right\},\quad N\ge3,\label{1.7}
\end{align}
equipping with the norm
\begin{align}
    \|\phi\|_{D^{1,2}(\mathbb{R}^N)}:=\left(\int_{\mathbb{R}^N}|\nabla\phi|^2\,{\rm{d}}x\right)^\frac12,\notag
\end{align}
we propose a unified method called \emph{monotonic approximation} method to investigate this singular problem by constructing the following set with a positive and monotone decreasing sequence $\{\delta_n\}$ called \emph{safe distances},
\begin{align}
G^c(\delta_n)=\bigl\{x\in\mathbb{R}^N \mid 1\ge |\nabla\phi_0|^2\ge 1-\delta_n\bigr\},\quad \lim_{\delta_n\to0}G^c(\delta_n)=G^c(0),\notag
\end{align}
where $\phi_0$ is the minimizer of \eqref{1.5}. We prove that $G^c(0)=\{x\in\mathbb{R}^N \mid |\nabla\phi_0|^2=1\}$ is a Lebesgue null set. Therefore, $|\nabla\phi_0|$ is avoided being at the singularity boundary of the singular Born--Infeld type functional. On this basis, we first illustrate the mechanism of our monotonic approximation method using \eqref{1.4} as a concrete example in Section \ref{s3}. We then extend the results to general singular Born--Infeld type theories, and the concrete results are summarized in the following theorems. 

First of all, in Sections \ref{s3} and \ref{snew}, for the rational-function theory and general singular Born--Infeld type theories with truncation thresholds, the monotonic approximation method we proposed leads us to Lemmas \ref{lemma2.4} and \ref{thm1}, which concludes that the minimizers of \eqref{1.5} will never attain the singular boundary $|\nabla\phi|^2=1$, such that we obtain the following theorem:
\begin{theorem}\label{thm2}
For any singular Born--Infeld type theory with a truncation threshold $s\le\frac12$, the functional \eqref{1.5} always has a unique minimizer $\phi_0\in\mathcal{A}\backslash\{0\}$ with $|\nabla\phi_0|^2<1$.
\end{theorem}
    
Furthermore, under the assumption that the charge distribution $\rho$ is radially symmetric, we prove rigorously that the minimizer is a weak solution, and that this solution is unique in the class of radial functions. 
\begin{theorem}\label{T1}
    If $\rho\in\mathcal{A}^*$ is radially distributed, then the minimizer of \eqref{1.5} is also a unique weak solution of the  system \eqref{1.6}.
\end{theorem}

Before we give more results, we emphasize a fundamental conceptual and methodological distinction here, particularly in comparison with the classical Born--Infeld setting studied in \cite{Bonheure2016}. The Lagrangian of the classical Born--Infeld model, $1-\sqrt{1-|\nabla\phi|^2}$, remains bounded as $|\nabla\phi|^2\to 1^-$; consequently, its energy functional does not exhibit singular behavior near the gradient constraint. This lack of singularity allows for direct variational arguments and truncation techniques to analyze the behavior of minimizers at the boundary. Bonheure, d'Avenia and Pomponio \cite{Bonheure2016} proved that the minimizer is a weak solution under radial symmetry of $\rho$, while Haarala \cite{Haarala2020} achieved the same conclusion without symmetry but required the strong assumption $\rho\in L^p(\mathbb{R}^N)$ ($p>N$) to obtain a uniform strict spacelikeness estimate.

In sharp contrast, the Born--Infeld type functionals considered in this paper are fundamentally singular: by the definition in \eqref{1.5}, their Lagrangian densities blow up as $s\to(1/2)^-$, i.e., as $|\nabla\phi|^2\to 1^-$. This strong singularity renders the methods of \cite{Bonheure2016} inapplicable for analyzing whether the gradient of the minimizer touches the singular boundary. Moreover,  the lack of continuity or weak lower semicontinuity of the functional \eqref{1.5} makes it difficult to analyze the existence of the minimizers. To overcome these difficulties, by using the novel monotonic approximation method we mentioned before, our Lemma \ref{thm1} establishes, for arbitrary $\rho\in\mathcal{A}^*$ and for general singular Born--Infeld type theories, that the singular set $\{|\nabla\phi_0|^2=1\}$ is a Lebesgue null set. This qualitative result is purely variational and does not require symmetry. Yet, a null-set alone does not suffice to derive the Euler--Lagrange equation in the non-radial case; for that, we impose radial symmetry in Theorem \ref{T1}, which allows us to pass from the variational inequality to the weak formulation via one-dimensional test functions. Thus, the conclusion $|G^c(0)|=0$ and Theorem \ref{T1} are complementary:  $|G^c(0)|=0$ is a universal structural property of the minimizer, obtained by a new method tailored to strong singularity, while Theorem \ref{T1} provides the functional equation under the physically natural radial setting.

In addition to the aforementioned theorems, we have some additional conclusions. First, if we impose more conditions on $\rho$, we can obtain the regularity of the weak solution of \eqref{1.6}.
\begin{theorem}[$C^1(\mathbb{R}^N,\mathbb{R})$ regularity of the weak solution]\label{T2}
    Let $\rho\in\mathcal{A}^*$ be radially distributed. If $\rho\in L^d(\mathbb{R}^N)\cap L^e(B_\xi(0))$ with $d\ge1$, $e\ge N$ and $\xi>0$, then the weak solution $\phi_0$ of \eqref{1.6} is $C^1(\mathbb{R}^N,\mathbb{R})$. Specifically, the weak solution possesses the following properties:
    \begin{enumerate}
        \item $|\phi_0'(0)|=0$.
        \item $\phi_0$ is strictly spacelike: $|\phi_0'(0)|\le1-\epsilon$ for a small perturbation parameter $\epsilon>0$.
    \end{enumerate}
\end{theorem}

Second, basing on the assumption for $\rho$ in Theorem \ref{T2}, we can obtain the classical regularity of $\phi_0$ if $\rho\in C^0$.
\begin{theorem}[Classical regularity of weak solution]\label{T4}
Let $\rho\in\mathcal{A}^*$ be radially distributed. Assume additionally that $\rho\in C^0(0,\infty)$. Then the weak solution $\phi_0$ of \eqref{1.6} satisfies $\phi_0\in C^2(0,\infty)$ and fulfills the Euler--Lagrange equation
    \begin{align}
        -\frac{1}{r^{N-1}}\frac{{\rm d}}{{\rm d}r}
        \Bigl(b\bigl(\tfrac12\phi_0'(r)^2\bigr)\,\phi_0'(r)\,r^{N-1}\Bigr)
        =\rho(r),\quad\forall\,r>0.\label{1.10}
    \end{align}
If $\rho$ is continuous at the origin, then $\phi_0\in C^2(\mathbb{R}^N)$ and \eqref{1.10} holds for all $r\ge 0$.
\end{theorem}

Lastly, under the premise that the minimizers of \eqref{1.5} never attain the singular boundary $s=\frac12$, the perturbation parameter $\epsilon$ defined in Theorem \ref{T2} quantitatively describes spacelikeness of $|\phi'_0|$ by some physical quantities.

\begin{theorem}[Uniform strict spacelikeness estimate]\label{T3}
    The small perturbation parameter $\epsilon>0$ given  by Theorem \ref{T2} depends explicitly on the spatial dimension $N$, charge density $\rho\in{L^e(B_\xi(0))}\cap L^d(\mathbb{R}^N)$ and spatial coordinates $[r_0,R]$.
\end{theorem}

The rest of the paper is organized as follows. Section \ref{s3} presents the structural assumptions on $\mathcal{L}(s)$ and $\mathcal{A}$, illustrates the monotonic approximation method on the rational-function theory, and analyzes the existence and uniqueness of the minimizer of \eqref{rational functional}. Section \ref{snew} extends these results to general singular Born--Infeld type theories with truncation thresholds and proves that any weak solution of \eqref{1.6} coincides with the minimizer of \eqref{1.5}. Section \ref{s4} treats the radial case: we prove that the minimizer of \eqref{1.5} is the weak solution of \eqref{1.6}, establish its $C^1$ and $C^2$ regularity, and provide a quantitative estimate for the strict spacelikeness of the minimizer via the perturbation parameter $\epsilon$. Finally, Section \ref{s5} synthesizes the main results and discusses their broader variational and physical context.

\section{The monotonic approximation method: A case study of the rational-function theory}\label{s3}

In this section, for the rational-function theory \eqref{1.4}, we demonstrate how to use the monotonic approximation method to prove that the minimizer does not attain the singular boundary, and then establish the existence and uniqueness of the minimizer.

We list some of the basic properties that are necessary and useful for this paper. 

First, it is necessary to summarize and explain the properties of the Born--Infeld type Lagrangians. We emphasize that,  these properties not only cover the properties of rational-function theory, but also will play a crucial role in the derivations in the subsequent sections for singular Born--Infeld type theories. 
\begin{proposition}\label{proposition2.1}
From the perspective of theoretical physics, any Born--Infeld type theory $\mathcal{L}(s)$ must satisfy the following signatures:
    \begin{enumerate}
        \item Normalization conditions \eqref{1.3}.
        \item According to \cite{Plebanski1968,Schellstede2016,Russo2024}, any Born--Infeld (type) theory belongs to the Pleba\'nski class of nonlinear electrodynamics and satisfies the causality conditions for all allowed background fields. Conditions (62)-(64) of \cite{Schellstede2016} give us
        \begin{align}
            b(s)=\mathcal{L}'(s)>0,\quad \mathcal{L}''(s)>0.\label{2.1}
        \end{align}
        \item By the Stone--Weierstrass theorem \cite{Stone1948,Yosida1995}, polynomial functions are dense in the space of continuous nonlinear Lagrangians. Consequently, any Born--Infeld type theory can be approximated, in an appropriate topology, by a sequence of polynomial-generated theories:
\begin{align}
    \mathcal{L}(s)\sim\sum_{n=1}^\infty \zeta_n s^n,\quad\zeta_n>0,\quad\zeta_n=\text{constant},\quad\zeta_1=1.\label{stone}
\end{align}
    \end{enumerate}
\end{proposition}

Here we give two concrete consequences deduced by Proposition \ref{proposition2.1}. 
\begin{enumerate}
    \item By \eqref{stone}, in this paper we employ the following scaling form controlled by a Taylor series:
\begin{align}
    \frac12|\nabla\phi|^2=s\le \mathcal{L}(s)\le s+\sum_{p=2}^\infty\alpha_{p} s^p=\sum_{p=1}^\infty\frac{\alpha'_{p}|\nabla\phi|^{2p}}{2^p},\quad \alpha'_{p}\in(0,1],\quad \alpha'_{p}={\rm constant}.\label{2.3}
\end{align}
\item  By \eqref{2.1}, $\mathcal{L}(\cdot)$ satisfies the following convex inequalities:
\begin{align}
    b(s_2)(s_2-s_1)\ge \mathcal{L}(s_2)-\mathcal{L}(s_1)\ge b(s_1)(s_2-s_1).\label{b2.2}
\end{align}
\end{enumerate}
It is easy to see that \eqref{1.4} satisfies \eqref{2.1}-\eqref{b2.2} apparently.

For the functional space \eqref{1.7}, we utilize the conclusions in \cite{Bonheure2016}.
\begin{lemma}
    The following assertions hold:
    \begin{enumerate}
        \item $\mathcal{A}$ is continuously embedded in $W^{1,p}(\mathbb{R}^N)$ for all $p \ge 2^* = \frac{2N}{N-2}$;
        \item $\mathcal{A}$ is continuously embedded in $L^\infty(\mathbb{R}^N)$;
        \item if $\phi \in \mathcal{A}$, then $\displaystyle \lim_{|x| \to \infty} \phi(x) = 0$;
        \item $\mathcal{A}$ is weakly closed;
        \item if $(\phi_n)_n \subset \mathcal{A}$ is bounded, there exists $\bar{\phi} \in \mathcal{A}$ such that, up to a subsequence, $\phi_n \to \bar{\phi}$ weakly in $\mathcal{A}$ and uniformly on compact sets.
    \end{enumerate}
\end{lemma}

We note that the functional \eqref{1.5} is not everywhere Fr\'{e}chet differentiable on $\mathcal{A}$. Therefore, to define critical points of \eqref{1.5} we employ the concept of subdifferential, given by
\begin{definition}\label{definition 3.2}
    Let $X$ be a real Banach space and $\Psi:X\to(-\infty, +\infty]$ a convex, lower semicontinuous function. Denote
    \begin{align}
        D(\Psi)=\{ u\in X \mid \Psi(u)< +\infty\}\notag
    \end{align}
    as the effective domain of $\Psi$. For $u\in D(\Psi)$, the set
    \begin{align}
        \partial\Psi(u)=\left\{ u^*\in X^*\;\middle|\;\Psi(v)-\Psi(u)\ge\langle u^*,v-u\rangle,\;\forall v\in X\right\}\notag
    \end{align}
    is called the subdifferential of $\Psi$ at $u$. Furthermore, consider the functional $I=\Psi+\Phi$, where $\Phi\in C^1(X, \mathbb{R})$. If $-\Phi'(u) \in \partial \Psi(u)$, i.e.,
    \begin{align}
        \langle\Phi'(u), v-u\rangle+\Psi(v)-\Psi(u)\ge0,\quad\forall v\in X,\notag
    \end{align}
    then $u$ is called a critical point of $I$ in the weak sense.
\end{definition}

\begin{remark}
    In the present convex setting, the subdifferential condition in Definition \ref{definition 3.2} is equivalent to the minimality of $\phi_0$. Indeed, since $J$ is convex, $0\in\partial I(\phi_0)$ is equivalent to $I(\phi_0) \le I(\phi)$ for all $\phi \in \mathcal{A}$. Therefore, although we adopt the terminology of weak critical points for conceptual consistency with non-smooth analysis, our subsequent proofs rely solely on the minimization property and the associated variational inequality, without invoking subdifferential calculus.
\end{remark} 

Now, we use Definition \ref{definition 3.2} to analyze the behavior of the minimizer near the singular boundary and to derive a variational inequality in the case of \eqref{1.4}. 

\begin{lemma}\label{lemma2.4}
    Assume that $\phi_0$ is the minimizer of \eqref{rational functional}. There exists a perturbed sequence $\{\delta_n\}_{n\ge 1}$, called the safe distance, with $0<\delta_{n+1}\le\delta_n$, such that the set
    \begin{align}
        G^c(\delta_n)=\left\{x\in\mathbb{R}^N\;\middle|\;1\ge|\nabla\phi_0|^2\ge 1-\delta_n\right\}\notag
    \end{align}
    is a null set with respect to the Lebesgue measure. Moreover,
    \begin{align}
        \frac{|\nabla\phi_0|^2}{(1-|\nabla\phi_0|^2)^2}\in L^1(\mathbb{R}^N).\label{3.1}
    \end{align}
    In particular, for every $\psi\in\mathcal{A}$, the following variational inequality holds:
    \begin{align}
        \int_{\mathbb{R}^N}\frac{\nabla\phi_0\cdot(\nabla\phi_0-\nabla\psi)}{(1-|\nabla\phi_0|^2)^2}\,{\rm{d}}x\le\langle\rho,\phi_0-\psi\rangle.\label{3.2}
    \end{align}
\end{lemma}

\begin{proof}
    Since the functional space $\mathcal{A}$ is convex, for any $t\in[0,1]$ and $\psi\in\mathcal{A}$ we have $\phi_t=(1-t)\phi_0+t\psi\in\mathcal{A}$. The minimality $I_R(\phi_0)\le I_R(\phi_t)$ implies
    \begin{align}
        \frac12\int_{\mathbb{R}^N}\left(\frac{|\nabla\phi_0|^2}{1-|\nabla\phi_0|^2}-\frac{|\nabla\phi_t|^2}{1-|\nabla\phi_t|^2}\right)\,{\rm{d}}x\le t\langle\rho,\phi_0-\psi\rangle.\label{3.3}
    \end{align}
    In particular, taking $\psi=0$, \eqref{3.3} becomes
    \begin{gather}
        \int_{G}\mathcal{G}_R\,{\rm d}x+\int_{G^c}\mathcal{G}_R\,{\rm d}x\le \langle\rho,\phi_0\rangle,\quad
        \mathcal{G}_R:=\frac{(2-t)|\nabla\phi_0|^2}{2(1-|\nabla\phi_0|^2)\bigl(1-(1-t)^2|\nabla\phi_0|^2\bigr)}.\label{3.4}
    \end{gather}
    Since $\mathcal{G}_R$ is monotonically increasing with respect to $|\nabla\phi_0|^2$, we have
    \begin{align}
        \int_{G^c}\mathcal{G}_R\,{\rm d}x\ge \int_{G^c}\frac{(1-\delta_n)(2-t)\,{\rm{d}}x}{2\delta_n\bigl(\delta_n+t(t-2)(\delta_n-1)\bigr)}\ge\frac{1-\delta_n}{\delta_n^2}\,|G^c(\delta_n)|.\label{3.5}
    \end{align}
    The family $\{G^c(\delta)\}$ is monotone decreasing with respect to $\delta_n$, i.e., $G^c(\delta_{n+1})\subset G^c(\delta_n)$. As $n\to\infty$, we have $\delta_n\to0^+$ and
    \begin{gather}
        \lim_{\delta_n\to0^+}G^c(\delta_n)=\bigcap_{n=1}^\infty G^c(\delta_n)=G^c(0);\quad 
        G^c(0)=\{x\in\mathbb{R}^N\mid |\nabla\phi_0|^2=1\}.\label{3.6}
    \end{gather}
    Combining \eqref{3.5} with \eqref{3.6} yields
    \begin{align}
        \lim_{\delta_n\to0^+}\int_{G^c}\mathcal{G}_R\,{\rm d}x\ge |G^c(0)|\,\lim_{\delta_n\to0^+}\frac{1-\delta_n}{\delta_n^2}.\notag
    \end{align}
    Since $\langle\rho,\phi_0\rangle<\infty$ ($\rho\in\mathcal{A}^*$) and 
    \begin{align}
        \lim_{\delta_n\to0^+}\frac{1-\delta_n}{\delta_n^2}=+\infty,\notag
    \end{align}
    the first inequality in \eqref{3.4} forces $G^c(0)$ to be a null set. Consequently, $|\nabla\phi_0|$ cannot attain the singular boundary.

    Now, when $|G^c(\delta_n)|=0$ and $\psi=0$, inequality \eqref{3.3} gives
    \begin{align}
        \int_{\mathbb{R}^N}\frac{(2-t)|\nabla\phi_0|^2}{2(1-|\nabla\phi_0|^2)\bigl(1-(1-t)^2|\nabla\phi_0|^2\bigr)}\,{\rm d}x\le \langle\rho,\phi_0\rangle.\notag
    \end{align}
    Applying Fatou's lemma as $t\to0^+$ yields
    \begin{align}
        \frac{|\nabla\phi_0|^2}{(1-|\nabla\phi_0|^2)^2}\in L^1(\mathbb{R}^N),\notag
    \end{align}
    which establishes \eqref{3.1}.

    Next, for $\psi\neq0$ and $|G^c(\delta_n)|=0$, we have
    \begin{align}
        \int_{\mathbb{R}^N}\frac{(2-t)|\nabla\phi_0|^2-2(1-t)\nabla\phi_0\cdot\nabla\psi-t|\nabla\psi|^2}{2(1-|\nabla\phi_0|^2)\bigl(1-|(1-t)\nabla\phi_0+t\nabla\psi|^2\bigr)}\,{\rm d}x\le\langle\rho,\phi_0-\psi\rangle.\label{3.7}
    \end{align}
    Using the following estimates in \eqref{3.7}:
    \begin{gather}
        \bigl|(2-t)|\nabla\phi_0|^2-2(1-t)\nabla\phi_0\cdot\nabla\psi-t|\nabla\psi|^2\bigr|
        \le 2|\nabla\phi_0|^2+2|\nabla\phi_0||\nabla\psi|+|\nabla\psi|^2
        \le 3|\nabla\phi_0|^2+2|\nabla\psi|^2,\notag\\
        (1-|\nabla\phi_0|^2)\bigl(1-|(1-t)\nabla\phi_0+t\nabla\psi|^2\bigr)
        \ge \delta_0\,(1-|\nabla\phi_0|^2)^2,\notag\\
        |\nabla\phi_t|^2\le 1-\delta_0,\ \delta_0>0,\notag
    \end{gather}
    we obtain
    \begin{align}
        \left|\frac{(2-t)|\nabla\phi_0|^2-2(1-t)\nabla\phi_0\cdot\nabla\psi-t|\nabla\psi|^2}{(1-|\nabla\phi_0|^2)\bigl(1-|(1-t)\nabla\phi_0+t\nabla\psi|^2\bigr)}\right|
        \le C\left(\frac{|\nabla\phi_0|^2}{(1-|\nabla\phi_0|^2)^2}+\frac{|\nabla\psi|^2}{(1-|\nabla\phi_0|^2)^2}\right).\label{3.8}
    \end{align}
    The right-hand side of \eqref{3.8} belongs to $L^1(\mathbb{R}^N)$ because $\psi\in D^{1,2}(\mathbb{R}^N)$ and by the integrability established above. Therefore, letting $t\to0^+$ in \eqref{3.7}, we deduce
    \begin{align}
        \int_{\mathbb{R}^N}\frac{\nabla\phi_0\cdot(\nabla\phi_0-\nabla\psi)}{(1-|\nabla\phi_0|^2)^2}\,{\rm{d}}x\le\langle\rho,\phi_0-\psi\rangle,\quad\text{for all }\psi\in\mathcal{A}.\notag
    \end{align}
    This establishes the variational inequality \eqref{3.2}.
\end{proof}

In view of the proof of Lemma \ref{lemma2.4}, we find that, on the one hand, the analysis of the rational-function model reveals that the key ingredients--the monotonicity of $\mathcal{L}_{\rm rational}$ and the blow-up of its derivative $b_{\rm rational}(s)$ as $s\to (1/2)^-$--are sufficient to guarantee that the minimizer stays away from the singular boundary. On the other hand, to illustrate the {\em difference} between the method in \cite{Bonheure2016} and our monotonic approximation method when proving that $G^c(\delta_n)$ is a null set, we provide the following remark.
\begin{remark}\label{remark3.1}
    After taking $\psi=0$ in the convex combination $\phi_t$, \eqref{3.3} is rewritten as \eqref{3.4}, which provides us with a singular function $\mathcal{G}_R$ that explodes at $|\nabla\phi_0|=1$. Therefore, when we adopt the method of \cite{Bonheure2016}, we cannot obtain the measure of $G^c(0)$ from \eqref{3.4}. Fortunately, the nonlinear structure of the rational-function theory provides the monotonicity of $\mathcal{G}_R$ (which, in fact, originates from the deeper Plebański-class structure shown in \eqref{2.1}), enabling us to use the monotonic approximation method to find the lower bound of the integral of $\mathcal{G}_R$ with respect to the safety distance $\delta_n$. This leads to \eqref{3.5} and allows us to obtain that $G^c(0)$ is a null set as $\delta_n\to0$.
\end{remark}

Now we investigate the existence and uniqueness of the minimizer of \eqref{rational functional}. By \eqref{1.4},  \eqref{rational functional} enjoys the following properties.
\begin{lemma}\label{lemmaconvex}
    The functional $I_{\rm rational}:\mathcal{A}\to\mathbb{R}$ is\\
        1. convex;\\
        2. coercive and bounded below.
\end{lemma}

\begin{proof}
    1. Apparently, we have
    \begin{align}
        \mathcal{L}''_{\rm rational}(s)=\frac{4}{(1-2s)^3}>0,\notag
    \end{align}
    such that $J_{\rm rational}(\phi)$ is strictly convex. 

    2. In view of \eqref{2.3}, we obtain
    \begin{align}
        I_{\rm rational}\ge\frac12\|\phi\|^2_{D^{1,2}(\mathbb{R}^N)}-\|\rho\|_{\mathcal{A}^*}\|\nabla\phi\|_{L^2(\mathbb{R}^N)},\notag
    \end{align}
    establishing $I_{\rm rational}$ is coercive and bounded below.
\end{proof}

For our purposes, we also need to study the weak lower continuity (or continuity) of $I_{\rm rational}$. However, on the functional space \eqref{1.7}, $I_{\rm rational}$ is not Fr\'{e}chet differentiable. Fortunately, Lemma \ref{lemma2.4} provides us with an important conclusion regarding the minimizer, $|\nabla\phi_0|^2<1$, so we can restrict $|\nabla\phi|^2\le1-\epsilon$ for our subsequent discussion by defining an {\em effective} functional space 
\begin{align}
    \mathcal{A}_{\rm effective}=D^{1,2}(\mathbb{R}^N)\cap\left\{\phi\in C^{0,1}(\mathbb{R}^N)\,\big|\,|\nabla\phi|^2\le1-\epsilon\right\},\quad N\ge3,\quad \epsilon\in(0,1).\notag
\end{align}
Obviously, $\mathcal{A}_{\rm effective}\subset\mathcal{A}$. Next, we will derive the continuity and weak lower semicontinuity of $I_{\rm rational}$.
\begin{lemma}\label{lemmaeff}
    On the effective functional space, the functional $I_{\rm rational}:\mathcal{A}_{\rm effective}\to\mathbb{R}$ is \\
        1. continuous;\\
        2. weakly lower semicontinuous.
\end{lemma}

\begin{proof}
    1. Since $\rho\in\mathcal{A}^*$ (hence $\rho\in\mathcal{A}^*_{\rm effective}$), it suffices to prove that $J_{\rm rational}(\phi)$ is continuous with $|\nabla\phi|\le1-\epsilon$. Let $\phi_n\to\phi$ in $\mathcal{A}$ and suppose there exists $w\in L^1(\mathbb{R}^N)$ such that $|\nabla\phi|^2\le w,\;|\nabla\phi_n|^2\le w$ a.e. in $\mathbb{R}^N$. Then
    \begin{align}
        \Big|\mathcal{L}_{\rm rational}\left(\frac12|\nabla\phi_n|^2\right)-\mathcal{L}_{\rm rational}\left(\frac12|\nabla\phi|^2\right)\Big|\le\frac12\sum_{p=1}^\infty\left(|\nabla\phi_n|^{2p}+|\nabla\phi|^{2p}\right),\notag
    \end{align}
    where we have used $\mathcal{L}_{\rm rational}(s)=\sum_{p=1}^\infty2^{p-1}s^p$. To proceed, by the interpolation inequality
    \begin{align}
        \|\nabla\phi\|_{L^{2p}(\mathbb{R}^N)}\le\|\nabla\phi\|^{\frac1p}_{L^2(\mathbb{R}^N)}\|\nabla\phi\|^{1-\frac1p}_{L^\infty(\mathbb{R}^N)},\notag
    \end{align}
    we have $|\nabla\phi_n|^{2p}\in L^1$ and $|\nabla\phi|^{2p}\in L^1$ for any $p\ge1$, and hence Lebesgue's Dominated Convergence Theorem yields $J_{\rm rational}(\phi_n)\to J_{\rm rational}(\phi)$.

    2. Since $J_{\rm rational}$ is continuous and strictly convex, while $\rho$ is continuous with respect to weak convergence, the weak lower semicontinuity of $I_{\rm rational}$ follows immediately.
\end{proof}

With the above three lemmas \ref{lemma2.4}-\ref{lemmaeff} at hand, we are now in a position to establish the existence and uniqueness of the minimizer for the rational-function functional \eqref{rational functional}.

Indeed, Lemmas \ref{lemmaconvex} and \ref{lemmaeff} establish the fundamental properties of the functional $I_{\mathrm{rational}}$: it is strictly convex, coercive, and weakly lower semicontinuous on $\mathcal{A}$. Consequently, if a minimizer $\phi_0$ exists in $\mathcal{A}$, it must be unique. Lemma \ref{thm1} further guarantees that any such minimizer cannot attain the singular boundary; that is, $|\nabla\phi_0|^2 \le 1-\epsilon$ for some $\epsilon>0$, ensuring that $\phi_0$ lies in the effective subspace $\mathcal{A}_{\mathrm{effective}}$.

On the other hand, restricting $I_{\mathrm{rational}}$ to $\mathcal{A}_{\mathrm{effective}}$ removes the singularity, making the functional continuous, strictly convex, and coercive. By the classical direct method in the calculus of variations, there exists a unique minimizer $\phi_\epsilon \in \mathcal{A}_{\mathrm{effective}}$. Since $\phi_\epsilon$ is a candidate for a minimizer on the full space $\mathcal{A}$, and any minimizer on $\mathcal{A}$ must belong to $\mathcal{A}_{\mathrm{effective}}$ by Lemma \ref{thm1}, the uniqueness on $\mathcal{A}_{\mathrm{effective}}$ forces $\phi_0 = \phi_\epsilon$. Therefore, the rational-function functional \eqref{rational functional} admits a unique minimizer $\phi_0 \in \mathcal{A}$, and this minimizer stays strictly away from the singular boundary.

\begin{remark}
    We emphasize that $\epsilon$ is a perturbation parameter that will be studied in Section \ref{s4}, and it is used to measure the distance between $|\nabla\phi_0|$ and the singular boundary in \eqref{1.7}.
\end{remark}

\section{General singular Born--Infeld type theories: Existence, uniqueness, and variational properties}\label{snew}

The preceding analysis of the rational-function theory establishes more than just the validity of the monotonic approximation method for that particular model. It reveals a structural principle: whenever the Lagrangian is monotone and blows up at the truncation threshold, the minimizer automatically stays away from the singular boundary, and this avoidance property in turn enables the direct method to yield existence and uniqueness of the minimizer. Since these two structural features are guaranteed by Proposition \ref{proposition2.1} for the entire class of singular Born--Infeld type theories, the same principle applies uniformly to all theories of the form \eqref{1.5}. We now develop this extension systematically, proving that the minimizer never touches the singular boundary, that \eqref{1.5} is continuous and weakly lower semicontinuous on $\mathcal{A}_{\rm effective}$, and that the unique minimizer exists for any $\rho\in\mathcal{A}^*$. These results will establish Theorem \ref{thm2} in full generality; additionally, we will show that any weak solution of \eqref{1.6} necessarily coincides with the minimizer (Lemma \ref{lemma 3.5}). The following lemma is the first and most essential step.

\begin{lemma}\label{thm1}
    Assume that $\phi_0$ is the minimizer of \eqref{1.5}. In the setting of an electrostatic field and the truncation threshold $|\nabla\phi|^2\le1$ in $\mathcal{L}(s)$, the following facts hold true for any Born--Infeld type Lagrangian $\mathcal{L}(s)$:
    \begin{enumerate}
        \item For \eqref{1.5}, there exists a monotone decreasing positive sequence $\{\delta_n\}_{n\ge 1}$ with $\lim_{n\to\infty}\delta_n=0$, such that the set
    \begin{align}
        G^c(\delta_n)=\left\{x\in\mathbb{R}^N\;\middle|\;1\ge |\nabla\phi_0|^2\ge 1-\delta_n\right\}\notag
    \end{align}
    enjoys $\lim_{\delta\to0^+}G^c(\delta_n)=G^c(0):=\{x\in\mathbb{R}^N\,|\,|\nabla\phi_0|^2=1\}$, and $G^c(0)$ is a null set with respect to the Lebesgue measure. 
    \item The minimizer $\phi_0$ and the function $b(s)=\mathcal{L}'(s)$ satisfy
    \begin{align}
        b\left(\frac12|\nabla\phi_0|^2\right)|\nabla\phi_0|^2\in L^1(\mathbb{R}^N).\label{1.8}
    \end{align}
    \item The following variational inequality holds:
    \begin{align}
        \langle\rho,\phi_0-\psi\rangle\ge\int_{\mathbb{R}^N}b\left(\frac12|\nabla\phi_0|^2\right)\nabla\phi_0\cdot\bigl(\nabla\phi_0-\nabla\psi\bigr)\,{\rm d}x,\quad\text{for any}\,\,\psi\in\mathcal{A}.\label{1.9}
    \end{align}
    \end{enumerate}
\end{lemma}
\begin{proof}
    For $t\in[0,1]$ and $\psi\in\mathcal{A}$, set $\phi_t=(1-t)\phi_0+t\psi\in\mathcal{A}$. The minimality $I(\phi_0)\le I(\phi_t)$ implies
    \begin{align}
        \langle\rho,\phi_0-\psi\rangle\ge\frac1t\int_{\mathbb{R}^N}\left[\mathcal{L}\left(\frac12|\nabla\phi_0|^2\right)-\mathcal{L}\left(\frac12|\nabla\phi_t|^2\right)\right]{\rm d}x
        :=\int_{\mathbb{R}^N}\mathcal{G}\,{\rm d}x
        =\left(\int_G+\int_{G^c}\right)\mathcal{G}\,{\rm d}x,\label{3.9}
    \end{align}
    where $G$ and $G^c$ denote the sets where $|\nabla\phi_0|^2$ is bounded away from $1$ and close to $1$, respectively.

    \textbf{Step 1.  Proof that $G^c(\delta_n)$ is a null set.}  
    Take $\psi=0$ and restrict to the set $G^c(\delta_n)$. Using the expansion \eqref{2.3} and the monotonicity properties derived from \eqref{2.1}, we obtain for $|\nabla\phi_0|^2\ge 1-\delta_n$,
    \begin{align}
       \mathcal{G}
       &\ge\frac1t\left[\mathcal{L}\left(\frac{1-\delta_n}{2}\right)-\mathcal{L}\left(\frac{(1-t)^2(1-\delta_n)}{2}\right)\right]\notag\\
       &\ge\lim_{t\to1^-}\frac1t\left[\mathcal{L}\left(\frac{1-\delta_n}{2}\right)-\mathcal{L}\left(\frac{(1-t)^2(1-\delta_n)}{2}\right)\right]\notag\\
       &=\mathcal{L}\left(\frac{1-\delta_n}{2}\right),\label{c4.2}
    \end{align}
    where, by \eqref{1.3}, \eqref{2.1} and \eqref{b2.2}, we have utilized the monotonicity of the following functions:
    \begin{align}
        \dfrac{{\rm d}[\mathcal{L}(z)-\mathcal{L}((1-t)^2z)]}{{\rm d}z}&=b(z)-(1-t)^2b((1-t)^2z)\ge0,\notag\\
        \dfrac{{\rm d}\left\{\frac1t\left[\mathcal{L}\left(z\right)-\mathcal{L}\left({(1-t)^2}z\right)\right]\right\}}{dt}&=\dfrac{1}{t^2}\left[\mathcal{L}\left({(1-t)^2}z\right)-\mathcal{L}\left(z\right)+2z(t-t^2)b\left({(1-t)^2}z\right)\right]\notag\\
        &\le-zb\left({(1-t)^2}z\right)\notag\\
        &<0,\notag
    \end{align}
    where $z=\frac{1-\delta_n}{2}$. Since $\mathcal{L}\left(\frac{1-\delta_n}{2}\right)$ is non-negative and tends to $+\infty$ as $n\to\infty$ (hence $\delta_n\to0^+$), while the integrand $\mathcal{G}$ remains bounded on the complementary set $G$ (where $|\nabla\phi_0|^2\le 1-\delta_n$). Hence, from \eqref{3.8}-\eqref{c4.2} we deduce
    \begin{align}
        \langle\rho,\phi_0\rangle \ge \frac1t\int_{G}\left[\mathcal{L}\left(\frac12|\nabla\phi_0|^2\right)-\mathcal{L}\left(\frac12|\nabla\phi_t|^2\right)\right]{\rm d}x+\mathcal{L}\left(\frac{1-\delta_n}{2}\right)|G^c(\delta_n)|.\label{b3.4}
    \end{align}
    Letting first $t\to0$ and then $\delta_n\to0$, the left-hand side stays finite while $\mathcal{L}\left(\frac{1-\delta_n}{2}\right)\to\infty$; therefore $G^c(\delta_n)\to G^c(0)$. In particular, $|G^c(0)|=0$, where $G^c(0)=\{x\in\mathbb{R}^N\,|\,|\nabla\phi_0|^2=1\}$.

    \textbf{Step 2.  Proof of the $L^1$ integrability \eqref{1.8}.}  
    Setting $\psi=0$ in \eqref{3.8} gives
    \begin{align}
        \langle\rho,\phi_0\rangle \ge \frac1t\int_{\mathbb{R}^N}\left[\mathcal{L}\left(\frac12|\nabla\phi_0|^2\right)-\mathcal{L}\left(\frac{(1-t)^2}{2}|\nabla\phi_0|^2\right)\right]{\rm d}x.\notag
    \end{align}
    The integrand is nonnegative by \eqref{2.1}. Applying Fatou's lemma as $t\to0^+$ yields
    \begin{align}
        \int_{\mathbb{R}^N}b\left(\frac12|\nabla\phi_0|^2\right)|\nabla\phi_0|^2\,{\rm d}x
        \le \liminf_{t\to0}\int_{\mathbb{R}^N}\frac{\mathcal{L}\left(\frac12|\nabla\phi_0|^2\right)-\mathcal{L}\left(\frac{(1-t)^2}{2}|\nabla\phi_0|^2\right)}{t}\,{\rm d}x
        \le \langle\rho,\phi_0\rangle<\infty,\notag
    \end{align}
    which establishes \eqref{1.8}.

    \textbf{Step 3.  Derivation of the variational inequality \eqref{1.9}.}  
    Now fix $\psi\neq0$. Using the convexity of $\mathcal{L}(s)$, we rewrite \eqref{3.8} as
    \begin{align}
        \langle\rho,\phi_0-\psi\rangle \ge \frac1t\int_{\mathbb{R}^N}\left[\mathcal{L}\left(\frac12|\nabla\phi_0|^2\right)-\mathcal{L}\left(\frac12|\nabla\phi_t|^2\right)\right]{\rm d}x. \label{b2.16}
    \end{align}
    We claim that the integrand is uniformly integrable. Indeed, by \eqref{b2.2} we have
    \begin{align}
        \left|\frac1t\left(\mathcal{L}\left(\frac12|\nabla\phi_0|^2\right)-\mathcal{L}\left(\frac12|\nabla\phi_t|^2\right)\right)\right|&\le \frac12\,b\!\left(\frac12|\nabla\phi_0|^2\right)\Bigl|(2-t)|\nabla\phi_0|^2-2(1-t)\nabla\phi_0\cdot\nabla\psi-t|\nabla\psi|^2\Bigr| \notag\\
        &\le C\,b\!\left(\frac12|\nabla\phi_0|^2\right)\bigl(|\nabla\phi_0|^2+|\nabla\psi|^2\bigr) \notag\\
        &\le C\,b\!\left(\frac12|\nabla\phi_0|^2\right)|\nabla\phi_0|^2 + C'(\delta)\,|\nabla\psi|^2,\quad C,C'(\delta)>0,\notag
    \end{align}
    where the last line belongs to $L^1(\mathbb{R}^N)$ thanks to \eqref{1.8}, $\psi\in D^{1,2}(\mathbb{R}^N)$, and $|\nabla\phi_0|\le 1-\delta_n$ (since $|G^c(\delta_n)|=0$). By the Lebesgue's Dominated Convergence Theorem, we may pass the limit $t\to0^+$ inside the integral in \eqref{b2.16} to obtain
    \begin{align}
        \langle\rho,\phi_0-\psi\rangle \ge \int_{\mathbb{R}^N} b\!\left(\frac12|\nabla\phi_0|^2\right)\nabla\phi_0\cdot(\nabla\phi_0-\nabla\psi)\,{\rm d}x,\notag
    \end{align}
    which is precisely \eqref{1.9}. This completes the proof.
\end{proof}

In view of the proofs in Lemma \ref{thm1} and Lemma \ref{lemma2.4}, we give the following remark.
\begin{remark}\label{remark3.2}
    By comparing the proofs of Lemma \ref{lemma2.4} and Lemma \ref{thm1}, we observe some interesting facts:
    \begin{enumerate}
        \item The observation in Remark \ref{remark3.1} extends verbatim to general singular Born--Infeld type theories $\mathcal{L}(s)$, provided that the monotonicity in Proposition \ref{proposition2.1} and blow-up properties in \eqref{1.5} hold. Indeed, since $\mathcal{G}$ in \eqref{3.9} is monotone with respect to $|\nabla\phi_0|^2$ and $\mathcal{L}(\frac{1-\delta_n}{2})\to\infty$ as $\delta_n\to0$, the same safe-distance argument yields $\lim_{\delta_n\to0}|G^c(\delta_n)|=|G^c(0)|=0$. The only difference from the rational-function case is that the explicit lower bound in \eqref{b3.4} now involves the general $\mathcal{L}(s)$ instead of the specific form $\mathcal{L}_{\rm rational}$, such that Lemma \ref{thm1} is valid for any singular Born--Infeld type theory.
        
        Here we present another example: the logarithmic model\cite{log1,log2}
    \begin{align}
        \mathcal{L}_{\log}\left(\frac12|\nabla\phi|^2\right)=-\frac12\ln(1-|\nabla\phi|^2),\quad b\left(\frac12|\nabla\phi|^2\right)=\frac{1}{1-|\nabla\phi|^2}.\label{log model}
    \end{align}
    On the one hand, \eqref{log model} obviously satisfies the structural assumptions of Proposition \ref{proposition2.1}: it is strictly convex, monotonically increasing, and blows up as $|\nabla\phi|^2\to1^-$. On the other hand, combining \eqref{log model} with \eqref{3.9}, we find that
    \begin{align}
        \mathcal{G}_{\rm log}:=\mathcal{L}_{\log}\left(\frac12|\nabla\phi_0|^2\right)-\mathcal{L}_{\log}\left(\frac12|\nabla\phi_t|^2\right)\ge\mathcal{L}_{\log}\left(\frac{1-\delta_n}{2}\right)\to+\infty,\quad \text{as }\delta_n\to0^+.
    \end{align}
    This inequality indicates $|G^c(0)|=0$ in the setting of \eqref{log model}. Hence all the conclusions of Lemma \ref{thm1} hold for this model as well, revealing a flexibility of our approach: it does not rely on any specific algebraic form of $\mathcal{L}(s)$, but only on the qualitative structural properties Proposition \ref{proposition2.1} and $\mathcal{L}(s)\to\infty$ as $s\to\frac12$.
    \item It is easy to check that, the conclusions \eqref{1.8} and \eqref{1.9} not only reinstate those of rational-function electrodynamics, \eqref{3.1}-\eqref{3.2}, but also include those of classical Born--Infeld theory\cite{Bonheure2016}.
    \end{enumerate}
\end{remark}

Now, we are ready to establish the existence and uniqueness of the minimizer of \eqref{1.5}. We can see that since functional $I$ inherently possesses singularity, when analyzing its continuity (or weak lower semi-continuity), Lemma \ref{thm1} must be borrowed. For any $\rho\in\mathcal{A}^*$, the following assertions of Born--Infeld type functional \eqref{1.5} hold.
\begin{lemma}\label{lemma2.1}
The functional $I:\mathcal{A}\to\mathbb{R}$ is\\
        1. convex;\\
        2. coercive and bounded below.
        
Moreover, since the minimizer cannot be on the singular boundary of $\mathcal{A}$ by Lemma \ref{thm1}, we also use the effective functional space $\mathcal{A}_{\rm effective}\subset\mathcal{A}$ defined in Section \ref{s3} to derive that the functional $I:\mathcal{A}_{\rm effective}\to\mathbb{R}$ is \\
        3. continuous;\\
        4. weakly lower semicontinuous.
\end{lemma}

\begin{proof}
    1. In light of \eqref{2.1}, we know that, for any Born--Infeld (type) Lagrangian, $J(\phi)$ is strictly convex. 

    2. In view of \eqref{2.3}, we obtain
    \begin{align}
        I\ge\frac12\|\phi\|^2_{D^{1,2}(\mathbb{R}^N)}-\|\rho\|_{\mathcal{A}^*}\|\nabla\phi\|_{L^2(\mathbb{R}^N)},\notag
    \end{align}
    revealing $I$ is coercive and bounded below.
    
    3. Since $\rho\in\mathcal{A}^*_{\rm effective}$, it suffices to prove that $J(\phi)$ is continuous with $|\nabla\phi|^2\le1-\epsilon$. Let $\phi_n\to\phi$ in $\mathcal{A}_{\rm effective}$ and suppose there exists $w\in L^1(\mathbb{R}^N)$ such that $|\nabla\phi|^2\le w,\;|\nabla\phi_n|^2\le w$ a.e. in $\mathbb{R}^N$. Then
    \begin{align}
        \Big|\mathcal{L}\left(\frac12|\nabla\phi_n|^2\right)-\mathcal{L}\left(\frac12|\nabla\phi|^2\right)\Big|&\le\Big|\mathcal{L}\left(\frac12|\nabla\phi_n|^2\right)\Big|+\Big|\mathcal{L}\left(\frac12|\nabla\phi|^2\right)\Big|\le \sum_{p=1}^\infty\frac{\alpha'_{p}}{2^p}\left(|\nabla\phi_n|^{2p}+|\nabla\phi|^{2p}\right),\notag
    \end{align}
    where we have used \eqref{2.3}. By the interpolation inequality
    \begin{align}
        \|\nabla\phi\|_{L^{2p}(\mathbb{R}^N)}\le\|\nabla\phi\|^{\frac1p}_{L^2(\mathbb{R}^N)}\|\nabla\phi\|^{1-\frac1p}_{L^\infty(\mathbb{R}^N)},\notag
    \end{align}
    we obtain $|\nabla\phi_n|^{2p}\in L^1$ and $|\nabla\phi|^{2p}\in L^1$ for any $p\ge1$. Applying Lebesgue's Dominated Convergence Theorem, we yield $J(\phi_n)\to J(\phi)$.

    4. Combining the continuity and strict convexity of $J$ with the weak continuity of $\rho$ yields the weak lower semicontinuity of $I$.
\end{proof}

Now we prove Theorem \ref{thm2}.

\emph{Proof of Theorem \ref{thm2}.}
\begin{proof}
    By Lemmas \ref{thm1} and \ref{lemma2.1}, the functional $I$ \eqref{1.5} has a unique minimizer $\phi_0\in\mathcal{A}$.
    
    It remains to show that $\phi_0\neq 0$. If the minimizer is trivial, then by \eqref{1.3} we have $I(0)=0$. Hence it is enough to show that there exists $\phi\in\mathcal{A}$ such that $I(\phi)<0$. Since $\mathcal{A}$ is convex, for any nonzero $\phi\in\mathcal{A}$ there exists $t>0$ sufficiently small such that $t\phi\in\mathcal{A}$. We compute
    \begin{align}
        \frac1t I(t\phi)
        &\le \sum_{p=1}^{\infty}\frac{\alpha'_{p}t^{2p-1}}{2^p}\|\phi\|_{D^{1,2}(\mathbb{R}^N)}^{2p}-\langle\rho,\phi\rangle.
    \end{align}
    As $t\to0^+$, the right-hand side tends to $-\langle\rho,\phi\rangle$. Choosing $\phi$ such that $\langle\rho,\phi\rangle>0$, we obtain $I(t\phi)<0$ for sufficiently small $t$. Thus $m<0=I(0)$, which implies that the minimizer cannot be the trivial function. This completes the proof.

    Consequently, we establish Theorem \ref{thm2} in full generality. 
\end{proof}
Note that both the rational-function theory (studied in Section \ref{s3}) and the logarithmic model (discussed in Remark \ref{remark3.2}) fall within the scope of this theorem.

Next, we analyze the relations and properties of minimizers and weak solutions. We first give the definition of the weak solution of \eqref{1.6} as follows.
\begin{definition}\label{weak solution}
    For a $\rho\in\mathcal{A}^*$, the weak solution $\phi$ of \eqref{1.6} is defined by
    \begin{align}
        \int_{\mathbb{R}^N}b\left(\frac12|\nabla\phi|^2\right)\nabla\phi\cdot\nabla\psi\,{\rm{d}}x=\langle\rho,\psi\rangle,\label{2.17}
    \end{align}
    for all $\psi\in\mathcal{A}$.
\end{definition}

To prepare the passage from the variational inequality \eqref{1.9} to the weak formulation in next section, we need the following lemma.

\begin{lemma}\label{lemma4.5}
    Let $\rho\in\mathcal{A}^*$. For a sequence $(\psi_n)_n$ satisfying $\psi_n\to\psi$ in $D^{1,2}(\mathbb{R}^N)$ and $|\nabla\psi|\le C$ ($C>0$), the following holds for the minimizer $\phi_0$:
    \begin{align}
        \lim_{n\to\infty}\int_{\mathbb{R}^N}b\left(\frac12|\nabla\phi_0|^2\right)\nabla\phi_0\cdot\nabla\psi_n\,{\rm{d}}x
        =\int_{\mathbb{R}^N}b\left(\frac12|\nabla\phi_0|^2\right)\nabla\phi_0\cdot\nabla\psi\,{\rm{d}}x.\label{4.3}
    \end{align}
\end{lemma}

\begin{proof}
Thanks to $|G^c(\delta_n)|=0$, we know that $b(\cdot)$ is finite with respect to $|\nabla\phi_0|$. Then we have
    \begin{align}
        \int_{\mathbb{R}^N}b\left(\frac12|\nabla\phi_0|^2\right)\nabla\phi_0\cdot\left(\nabla\psi_n-\nabla\psi\right)\,{\rm d}x&\le C(\delta)\int_{\mathbb{R}^N}\nabla\phi_0\cdot\left(\nabla\psi_n-\nabla\psi\right)\,{\rm d}x\notag\\
        &\le C(\delta)\|\phi_0\|_{D^{1,2}(\mathbb{R}^N)}\|\psi_n-\psi\|_{D^{1,2}(\mathbb{R}^N)}\to0.\notag
    \end{align}
    By Lebesgue’s Dominated Convergence Theorem we obtain \eqref{4.3}.
\end{proof}

\begin{remark}
    Lemma  \ref{thm1} establishes that the singular set is a null set for any $\rho\in\mathcal{A}^*$, without symmetry assumptions. However, to pass from the variational inequality \eqref{1.9} to Euler--Lagrange equation in the 
    strong sense, additional structure is needed. In Section \ref{s4} we restrict our attention to radially symmetric charge distributions, which constitute a physically natural and mathematically tractable framework.
\end{remark}

Lastly, in the following lemma, we prove the weak solution of \eqref{1.6} must be a minimizer of \eqref{1.5} for any Born--Infeld type theory $\mathcal{L}(s)$. Note that this lemma is obviously valid for \eqref{1.4} and \eqref{log model}.
\begin{lemma}\label{lemma 3.5}
    Let $\rho\in\mathcal{A}^*$. If $\phi\in\mathcal{A}$ is a weak solution of \eqref{1.6}, then $\phi=\phi_0$.
\end{lemma}

\begin{proof}
    Since $\phi$ is a weak solution, for any $\psi\in\mathcal{A}$ we have
    \begin{align}
        \int_{\mathbb{R}^N}b\left(\frac12|\nabla\phi|^2\right)\nabla\phi\cdot\nabla\psi\,{\rm{d}}x=\langle\rho,\psi\rangle.\notag
    \end{align}
    Taking $\psi=\phi$ and $\psi=\phi_0$ respectively, we obtain
    \begin{align}
        \int_{\mathbb{R}^N}b\left(\frac12|\nabla\phi|^2\right)|\nabla\phi|^2\,{\rm{d}}x=\langle\rho,\phi\rangle,\quad
        \int_{\mathbb{R}^N}b\left(\frac12|\nabla\phi|^2\right)\nabla\phi\cdot\nabla\phi_0\,{\rm{d}}x=\langle\rho,\phi_0\rangle.\label{3.13}
    \end{align}
    Subtracting \eqref{3.13} and using the convexity of the functional, we get
    \begin{align}
        \langle\rho,\phi_0-\phi\rangle=\int_{\mathbb{R}^N}b\left(\frac12|\nabla\phi|^2\right)\nabla\phi\cdot(\nabla\phi_0-\nabla\phi)\,{\rm{d}}x\le\int_{\mathbb{R}^N}\mathcal{L}\left(\frac12|\nabla\phi_0|^2\right)-\mathcal{L}\left(\frac12|\nabla\phi|^2\right)\,{\rm d}x,\notag
    \end{align}
    which implies $I(\phi_0)\ge I(\phi)$. Hence the weak solution is precisely the minimizer.
\end{proof}

\section{General singular Born--Infeld type theories: From minimizers to weak solutions and regularity in the radial case}\label{s4}

Lemma \ref{thm1} establishes the structural properties of the minimizer for arbitrary $\rho\in\mathcal{A}^*$. However, to pass from the variational inequality \eqref{1.9} to the Euler--Lagrange equation, one needs to construct suitable test functions that can probe the singular set $G^c(0)$. In the non-radial case, the lack of symmetry prevents such a construction in general. Therefore, in this section we restrict to radially symmetric $\rho$, where the one-dimensional structure of the problem allows us to complete the proof. We note that all results established below apply in particular to the rational-function model \eqref{1.4} and logarithmic model \eqref{log model}.

It is worth mentioning that, in classical Born--Infeld theory, purely electrostatic configurations are not always compatible with non-symmetric sources when magnetic fields are consistently included in the full time-dependent theory \cite{Yang2024}. However, for radially symmetric charge distributions, the electrostatic ansatz is exact and self-consistent. This makes the radial setting a natural and physically meaningful starting point for developing the variational framework, before attempting more general configurations. Moreover, the radial reduction allows us to illustrate the mechanism of the monotonic approximation method in a transparent way and to obtain sharp regularity results. Extensions to non-radial sources will be discussed in future work.

We firstly give the definition of a radially distributed charge.

{\begin{definition}
    Let $\mathcal{A}$ be the function space defined in \eqref{1.7} and denote by $\mathcal{A}^*$ its dual.
For $\tau\in O(N)$ and $\phi\in\mathcal{A}$, define $\phi^\tau(x):=\phi(\tau x)$.
For $\rho\in\mathcal{A}^*$, define $\rho^\tau\in\mathcal{A}^*$ by
\begin{align}
\langle \rho^\tau,\phi\rangle := \langle \rho,\phi^{\tau^{-1}}\rangle,\quad\forall\phi\in\mathcal{A}.\notag
\end{align}
We say that $\rho$ is \emph{radially distributed} if $\rho^\tau=\rho$ for every $\tau\in O(N)$, and the radial-type of functional space is defined by
\begin{align}
    \mathcal{A}_{\rm radial}=\{\phi\in\mathcal{A}\,\big|\,\phi^\tau=\phi\,\,\text{for every}\,\,\tau\in O(N)\}.\notag
\end{align}
\end{definition}
Following the above definitions, we prove that radial minimizers are weak solutions.}

\medskip

{\emph Proof of Theorem \ref{T1}:}
\begin{proof}
    For any $\tau\in O(N)$, $\rho\in\mathcal{A}^*_{\rm radial}$ and $\rho^\tau=\rho$, we have $I(\phi^\tau_0)=I(\phi_0)$, so that $\phi_0\in\mathcal{A}_{\rm radial}$ by the uniqueness of minimum.

    Next we prove the minimizer $\phi_0(x)=\phi_0(r)$ with $r=|x|$ is the weak solution of \eqref{1.6}. Define 
    \begin{align}
        \mathcal{K}_{k_n}:=\left\{r\ge0\,\Big|\,1\ge|\phi_0'(r)|\ge1-\delta_n-\frac1{k_n}\right\},\quad k_n\ge(1-\delta_n)^{-1}.\notag
    \end{align}
    Notice that $\mathcal{K}_{k_n}$ is monotone decreasing with respect to $(\frac1{k_n}+\delta_n)$, so we have $|G^c(\delta_n)|=|\cap_{k_n\ge(1-\delta_n)^{-1}}\mathcal{K}_{k_n}|=0$ as $n\to\infty$.

Take the test function $\psi\in C^\infty_c(\mathbb{R}^N)\cap \mathcal{A}_{\rm radial}$ with support set $[0,R]$:
    \begin{align}
        \psi_{k_n}=-\int^\infty_r\psi'(z)[1-\chi_{{\cal K}_{k_n}(z)}]\,{\rm{d}}z,\quad\chi_{{\cal K}_k(z)}=\begin{cases}
            1,\quad z\in {\cal K}_{k_n},\\
            0,\quad z\notin{\cal K}_{k_n}.
        \end{cases}\notag
    \end{align}
    To proceed, for any $t$ small enough, $|\phi'_0+t\psi_{k_n}'|\le1-\delta_n$ as $r\in\mathcal{K}_{k_n}$, otherwise taking $t=(k_n\|\psi'_{k_n}\|_\infty)^{-1}$, we have
    \begin{align}
        |\phi'_0+t\psi'_k|\le|\phi_0'|+t\|\psi'\|_\infty\le1-\delta_n.\notag
    \end{align}
    Now, since $\phi_0$ is the minimizer of $I$ and $\psi'_{k_n}=\psi'(r)(1-\chi_{\mathcal{K}_{k_n}})$, we have 
    \begin{align}
        \lim_{t\to0}\frac{I(\phi_0+t\psi_{k_n})-I(\phi_0)}{t}&={\omega_N}\int^\infty_0b\left(\frac12(\phi'_0)^2\right){\phi_0'\psi'(1-\chi_{\mathcal{K}_{k_n}(s)})r^{N-1}\,{\rm{d}}r}-\omega_N\int_0^\infty\rho\psi_{k_n} r^{N-1}\,{\rm{d}}r\notag\\
        &=0.\notag
    \end{align}
    Since the set $\mathcal{K}_{k_n}$ enjoys 
    \begin{enumerate}
        \item $\mathcal{K}_{k_{n+1}}\subset\mathcal{K}_{k_n}$ as $k_{n+1}\ge k_n$; 
        \item $|\mathcal{K}_{k_n}|\to0$ as ${n}\to\infty$; 
        \item There exists $k_n$ such that as $n\to\infty$ we have $r\notin \mathcal{K}_{k_n}$ and $\chi_{\mathcal{K}_{k_n}}\to0$ a.e. in $r\in[0,\infty)$,
    \end{enumerate}
then we find that 
\begin{align}
    \int^\infty_0b\left(\frac12(\phi'_0)^2\right){\phi_0'\psi'(1-\chi_{\mathcal{K}_{k_n}(s)})r^{N-1}\,{\rm{d}}r}\to \int^\infty_0b\left(\frac12(\phi'_0)^2\right){\phi_0'\psi'r^{N-1}\,{\rm{d}}r}.\label{2.15}
\end{align}
On the other hand, in view of the traits of functional space $\mathcal{A}$, we conclude $\psi_{k_n}\to\psi$ in $\mathcal{A}$, such that
\begin{align}
    \langle\rho,\psi_{k_n}\rangle\to\langle\rho,\psi\rangle \quad\text{as} \quad n\to\infty.\label{2.16}
\end{align}
In view of \eqref{2.15}-\eqref{2.16}, we obtain the following weak formula,
\begin{align}
    \int_{\mathbb{R}^N}b\left(\frac12|\nabla\phi_0|^2\right){\nabla\phi_0\cdot\nabla\psi\,{\rm{d}}x}=\langle\rho,\psi\rangle,\notag
\end{align}
satisfying Definition \ref{weak solution} in the radial setting of $\rho$ and $\phi_0$.

Finally, we prove that for any $\psi \in \mathcal{A}_{\rm rad}$, the minimizer $\phi_0(r)$ is indeed a weak solution, without assuming $\psi$ to be $C_c^\infty$. To this end, we construct a sequence $(\psi_n)_n \subset C_c^\infty(\mathbb{R}^N)$ of radially symmetric, smooth, compactly supported functions such that $\psi_n \to \psi$ in $D^{1,2}(\mathbb{R}^N)$ and $\|\nabla \psi_n\|_\infty \le C$, where $C$ is not necessarily equal to $1$. Convolving with a smooth radially symmetric mollifier $\xi_n$ applied to $\chi_n \psi$, we set
\begin{align}
\psi_n = \xi_n * (\chi_n \psi),\notag
\end{align}
where $\chi_n(\cdot)= \chi(\cdot / n) : \mathbb{R}^N \to \mathbb{R}$ is a smooth radially symmetric function satisfying
\begin{align}
    \chi(x)=\begin{cases}
        1,\quad{\rm if}\,\,|x|\le1,\\
        0,\quad{\rm if}\,\,|x|\ge2.
    \end{cases}\notag
\end{align}
In light of Lemma \ref{lemma4.5}, it is suffice to deduce $\phi_0$ is a weak solution of \eqref{1.6} for $\psi\in\mathcal{A}_{\rm radial}$. Moreover, if we replace $\psi$ with $\phi_0$ in \eqref{2.17}, then $\phi_0$ is the weak solution for $\psi\in\mathcal{A}$.
\end{proof}

Having established that the radial minimizer is a weak solution, we now turn to its regularity properties. We deduce the regularity of $\phi_0$ by taking different assumptions of $\rho$. We firstly give a auxiliary lemma.
\begin{lemma}\label{lemma4.7}
    Assume $\rho \in L^d(\mathbb{R}^N)$ ($d \ge 1$) is radially symmetric. Then the weak solution $\phi_0$ of \eqref{1.6} satisfies $\phi_0 \in C^1([r_0, R])$ for every $0 < r_0 < R < \infty$; i.e., $\phi_0$ is $C^1$ away from the origin.
\end{lemma}

\begin{proof}
    For every $\psi\in\mathcal{A}_{\rm radial}$ and in the radial case, the weak formulation gives us
    \begin{align}
        +\infty>\int^\infty_0\rho\,\psi\,r^{N-1}{\rm{d}}r=\int^\infty_{0}b\left(\frac12(\phi'_0)^2\right){\phi_0'\cdot\psi'\, r^{N-1}\,{\rm{d}}x}:=\int^\infty_{0}\mathcal{B}\cdot\psi'{\rm{d}}r,\,\, \mathcal{B}=b\left(\frac12(\phi'_0)^2\right)\phi_0'r^{N-1}.\notag
    \end{align}
    Therefore, we find that, on the one hand, $\mathcal{B}'\in{L}^1(0,R)$ since $\rho\in L^d(\mathbb{R}^N)$. On the other hand, by utilizing \eqref{1.8} and defining $H=\{r\in(0,R)\,\big|\,|\phi_0'|^2\le\frac12\}$, we have
    \begin{align}
        \int^R_0\mathcal{B}\,{\rm d}r\le\int_{(0,R)\cap H}\mathcal{B}\,{\rm d}r+\int_{{(0,R)\cap H^c}}\mathcal{B}\,{\rm d}r\le C\left(\int_0^Rr^{N-1}\,{\rm{d}}r+\int^R_0b\left(\frac12(\phi'_0)^2\right)(\phi_0')^2r^{N-1}\,{\rm d}r\right)<\infty.\notag
    \end{align}
    Hence, $\mathcal{B}\in W^{1,1}(0,R)$, which is continuously embedded in $C([0,R])$.

    Consequently, $\mathcal{B}$ is continuous on $[0,R]$ and  there exist $\epsilon>0$ such that $|\phi_0'|\le1-\epsilon$ since
    \begin{align}
        \mathcal{B}\le C(R)\iff b\left(\frac12|\phi'_0|^2\right)\phi_0'r^{N-1}\le C(R)\notag
    \end{align}
    for some constant $C(R)>0$.
\end{proof}

To proceed, we assume that $\rho\in L^e(B_\xi(0))$ with $e\ge N$ and $\xi>0$, then the regularity of weak solution $\phi_0$ can be improved up to $r=0$.

\medskip

{\em Proof of Theorem \ref{T2}}.
\begin{proof}
    Taking the test function $\psi_R$ satisfies
    \begin{align}
        \psi_R=\begin{cases}
            R-|x|,&\quad |x|\le R;\\
            0,&\quad |x|>R,
        \end{cases}\notag
    \end{align}
    such that the weak formula in the radial case of charge is rewritten as
    \begin{align}
        -\frac1R\int^R_0\mathcal{B}\,{\rm d}r=-\frac1R\int^R_0b\left(\frac12(\phi'_0)^2\right)\phi_0'r^{N-1}\,{\rm }dr=\int^R_0\rho \,r^{N-1}\,{\rm d}r-\frac1R\int^R_0\rho\,r^N\,{\rm d}r.\label{4.9}
    \end{align}
    Since $\mathcal{B}$ is continuous shown by Lemma \ref{lemma4.7}, Mean Value Theorem gives us
    \begin{align}
        \lim_{R\to0^+}-\frac1R\int^R_0\mathcal{B}\,{\rm }dr=-b\left(\frac12(\phi_0'(R))^2\right)\phi_0'(R)R^{N-1}=\mathcal{B}(R).
    \end{align}
    In fact, the right-hand side of \eqref{4.9} will tend to $0$ as $R\to0^+$ by L'Hôpital's Rule, then $\mathcal{B}(R)=0$.

    Since the weak derivative of $\mathcal{B}$ equals to $-\rho(r)r^{N-1}$ and $\rho\in L^s,(s\ge1)$, we have
    \begin{align}
        \int^R_0\rho\,r^{N-1}\,{\rm d}r
=-\mathcal{B}(R)=0,\label{4.11}
\end{align}
such that 
\begin{align}
    b\left(\frac12(\phi_0'(R))^2\right)\phi_0'(R)\le& R^{1-N}\int^R_0|\rho|\,r^{N-1}\,{\rm d}r\notag\\
    \le&R^{1-N}\left(\int^R_0|\rho|^e r^{N-1}\,{\rm d}r\right)^\frac1e\cdot\left(\int^R_0r^{N-1}\,{\rm d}r\right)^{1-\frac1e}\notag\\
    \le&C R^{1-\frac Ne}\|\rho\|_{L^e(B_\xi(0))},\quad \text{for some }C>0\,\,\text{and }R<\xi.\notag
\end{align}
    Since $e>N$, so the last line of the above formula will tend to $0$ as  $R\to0^+$, leading to $\phi_0'(R)=0$ by \eqref{1.3}.
\end{proof}

Consequently, we conclude that $\phi_0\in C^1(\mathbb{R}^N,\mathbb{R})$. Actually, we can further assume $\rho\in C^0$ to enhance the regularity of the weak solution to be $C^2(\mathbb{R}^N,\mathbb{R})$.

{\em Proof of Theorem \ref{T4}}.

\begin{proof}
From the proof of Theorem \ref{T2} we recall the notation
    \begin{align}
        \mathcal{B}(r):=b\bigl(\tfrac12\phi_0'(r)^2\bigr)\,\phi_0'(r)\,r^{N-1},\notag
    \end{align}
which satisfies $\mathcal{B}\in C([0,R])$ for every $R>0$ by Lemma \ref{lemma4.7}. The weak formulation in the radial case reads
    \begin{align}
        \int_0^\infty \mathcal{B}(r)\,\psi'(r)\,{\rm d}r
        =\int_0^\infty \rho(r)\,\psi(r)\,r^{N-1}\,{\rm d}r,
        \quad\forall\,\psi\in C_c^\infty(0,\infty).\label{weak-B}
    \end{align}
    
Since $\mathcal{B}$ is continuous and the right-hand side of \eqref{weak-B} defines a locally integrable function $-\rho(r)r^{N-1}$, the one-dimensional equivalence between distributional and pointwise derivatives implies
    \begin{align}
        \mathcal{B}'(r)=-\rho(r)\,r^{N-1}\quad\text{a.e. in }(0,\infty).\label{B-prime}
    \end{align}
With the additional assumption $\rho\in C^0(0,\infty)$, the right-hand side of \eqref{B-prime} is continuous, whence $\mathcal{B}\in C^1(0,\infty)$.
    
Define the auxiliary function
    \begin{align}
        \Phi(t):=b\bigl(\tfrac12 t^2\bigr)\,t,\quad t\in[0,1).\notag
    \end{align}
By Proposition \ref{2.1}, $b(s)>0$ and $b'(s)>0$ for $s\in[0,\frac12]$, hence
    \begin{align}
        \Phi'(t)=b\bigl(\tfrac12 t^2\bigr)+b'\bigl(\tfrac12 t^2\bigr)\,t^2>0,
        \quad\forall\,t\in[0,1).\notag
    \end{align}
Consequently, $\Phi\in C^\infty[0,1)$ is a strictly increasing diffeomorphism onto its image, and its inverse $\Phi^{-1}$ is also $C^\infty$.
    
From the definition of $\mathcal{B}$ we have
    \begin{align}
        \phi_0'(r)=\Phi^{-1}\!\Bigl(\frac{\mathcal{B}(r)}{r^{N-1}}\Bigr),\quad r>0.\notag
    \end{align}
The function $r\mapsto r^{N-1}$ is smooth and non-vanishing on $(0,\infty)$. Since $\mathcal{B}\in C^1(0,\infty)$, we yield
    $\phi_0'\in C^1(0,\infty)$, i.e., $\phi_0\in C^2(0,\infty)$.
    
Differentiating $\mathcal{B}(r)=\Phi(\phi_0'(r))\,r^{N-1}$ and using 
    \eqref{B-prime}, we obtain for every $r>0$
    \begin{align}
        \Phi'(\phi_0'(r))\,\phi_0''(r)\,r^{N-1}+\Phi(\phi_0'(r))\,(N-1)\,r^{N-2}=-\rho(r)\,r^{N-1}.\notag
    \end{align}
Dividing by $r^{N-1}$ and recognizing the radial expression of the divergence operator gives precisely \eqref{1.10}.
    
It remains to discuss the regularity at the origin. From \eqref{4.11} we have
    \begin{align}
        \mathcal{B}(r)=-\int_0^r \rho(w)\,w^{N-1}\,{\rm d}w.\notag
    \end{align}
If $\rho$ is continuous at $0$, then $\rho(r)=\rho(0)+o(1)$ as $r\to0^+$, and
    \begin{align}
        \mathcal{B}(r)=-\frac{\rho(0)}{N}\,r^N+o(r^N),\quad r\to0^+.\notag
    \end{align}
Recall that $\Phi(t)=t+O(t^2)$ near $t=0$ because $b(0)=1$. Hence
    \begin{align}
        \phi_0'(r)=\Phi^{-1}\!\Bigl(\frac{\mathcal{B}(r)}{r^{N-1}}\Bigr)=-\frac{\rho(0)}{N}\,r+o(r),\quad r\to0^+.\notag
    \end{align}
This shows that $\phi_0'(0)=0$ (already known from Theorem \ref{T2}) and that
    \begin{align}
        \phi_0''(0)=\lim_{r\to0^+}\frac{\phi_0'(r)-\phi_0'(0)}{r}=-\frac{\rho(0)}{N}\notag
    \end{align}
exists and is finite. Hence, the continuity of $\phi_0''$ at $0$ follows from the continuity of $\rho$ and the equation \eqref{1.10} itself. Thus $\phi_0\in C^2(\mathbb{R}^N)$ and the Euler--Lagrange equation holds for all $r\ge 0$.
\end{proof}

To end up this section, it is noted that in Lemma \ref{lemma4.7}, we defined $\epsilon$ such that when $|\phi'_0|\le1-\epsilon$, $\phi_0\in C^1$. We will now analyze how this small perturbation depends on the charge density.

\medskip

{\em Proof of Theorem \ref{T3}.}
\begin{proof}
    Define 
    \begin{align}
        M:=-\frac{1}{r^{N-1}}\int^R_0\rho\,r^{N-1}\,{\rm d}r,\notag
    \end{align}
    such that \eqref{4.11} gives us
    \begin{align}
        \mathbb{B}(|\phi_0'|):=b\left(\frac12|\phi_0'|^2\right)|\phi_0'|=|M|.\label{4.12}
    \end{align}
    Note that the left-hand side of \eqref{4.12} is always monotonically increasing with respect to $\phi_0'$ by \eqref{1.3} and enjoys $\lim_{|\phi_0'|\to1^-}\mathbb{B}(\phi_0')=+\infty$, we conclude that the inverse function of $\mathbb{B}(|\phi_0'|)$, that is, $\mathbb{B}^{-1}\left(|M|\right)$, is also monotonically increasing and continuous(continuity is guaranteed by Lemma \ref{lemma4.7}). Therefore, from \eqref{4.12} we have
    \begin{align}
        |\phi_0'|\le\mathbb{B}^{-1}\left(|M|\right),\quad 0<|M|\le \frac{N^{e^{-1}-1}}{r_0^{N-1}}\|\rho\|_{L^e(B_\xi(0))}\cdot R^{N(1-e^{-1})}:=|M_0|\label{4.13}.
    \end{align}
    Taking $\epsilon=1-\mathbb{B}^{-1}\left(|{M}|\right)$, we know that the upper limit of $|\phi_0'|$ depends on the selection of spatial coordinates, spatial dimensions and charge density. 

    We finishes the proof.
\end{proof}

\begin{remark}
    There exist important observations for \eqref{4.13}: 
    \begin{enumerate}
        \item As $r_0\to0^+$, $\mathbb{B}^{-1}(|M|)$ will increase, reducing $\epsilon$ and leading to $|\phi_0'|\to1^-$ more and more closely.
        \item It is easy to see that $R$ also has a significant impact on $|\phi_0'|$ and $\epsilon$, implying the nature of the distant space will affect the spacelikeness of $\phi_0$ near the origin.
    \end{enumerate}
\end{remark}

Here we use the logarithmic model \eqref{log model} as an example for the construction of $\mathbb{B}^{-1}(|M|)$. This theory meets all the pre-established conditions for the Born--Infeld type theory as stipulated in this article. Therefore, it can seamlessly fit into our framework. 

In light of \eqref{4.12}, we have
    \begin{align}
        \mathbb{B}(|\phi_0'|)=\frac{|\phi_0'|}{{1-|\phi_0'|^2}},\notag
    \end{align}
    giving rise to
\begin{align}
    (\mathbb{B}^{-1}(|M|))^2=\frac{2}{1+\sqrt{1+4|M|^2}},\notag
\end{align}
which is obviously monotone increasing with respect to $|M|$. Therefore, the perturbation parameter of spacelikeness estimate of $|\phi_0'|$ enjoys
\begin{align}
    \mathbb{B}^{-1}_{\rm max}(|M|)=\left(\frac{2}{1+\sqrt{1+4|M|^2}}\right)^\frac12\le\epsilon<1,\notag
\end{align}
where $|M_0|$ is defined in \eqref{4.13}.

\section{Conclusion}\label{s5}

In this paper, in the setting of an electrostatic field in Maxwell action, that is, $s=\frac12{|\nabla\phi|^2}$, we develop a unified variational framework for a class of singular Born--Infeld type functionals \eqref{1.5} whose Lagrangians blow up at the constraint boundary $s=\frac12$ for the Born--Infeld type theory $\mathcal{L}(s)$ with truncation threshold $s\le\frac12$. By introducing a monotonic approximation method in Sections \ref{s3} and \ref{snew}, we have shown that the unique minimizer $\phi_0$ of the energy functional \eqref{1.5} avoids the singular boundary (Lemma  \ref{thm1} and Lemma \ref{lemma2.1}). We have summarized the signatures and advantages of this method in Remarks \ref{remark3.1} and \ref{remark3.2}. Furthermore, this structural property is essential to establish a variational inequality (see Lemma  \ref{thm1}) and to identify the minimizer as a weak solution to the associated Euler--Lagrange equation under radial symmetry (see Theorem \ref{T1}). In particular, the minimizer $\phi_0$ satisfies a strict spacelikeness condition almost everywhere (see Theorem \ref{T3}), this enables us to present the conclusion that $G^c(0)=\{x\in\mathbb{R}^N\,|\,|\nabla\phi_0|^2=1\}$ is a null set obtained through the monotonic approximation method in Lemma  \ref{thm1} in a quantitative manner. Moreover, we have obtained optimal $C^1$ and $C^2$ regularity results in Theorems \ref{T2} and \ref{T4} for the weak solution and provide an estimate of the perturbation parameter $\epsilon$ in the proof of Theorem \ref{T3}, which measures the distance from the singular barrier in terms of the physical data of the problem. These results extend the classical Born--Infeld theory to a broader class of singular models and lay a rigorous foundation for future studies on non-radial charge distributions and other related nonlinear field equations.

Beyond their purely mathematical interest, the system \eqref{1.6} carries a deep physical motivation. The Born--Infeld type Lagrangian, with its characteristic constraint, not only serves as a generalization of classical electrodynamics but also plays a fundamental role in cosmology and string theory. In particular, the effective dynamics of {\em tachyon} fields in string theory and Born--Infeld inspired cosmological models are governed by actions of the Dirac--Born--Infeld form, where the constraint $|\nabla\phi|\le1$ emerges as a natural consistency condition for the causal structure of the theory \cite{Padmanabhan2002,Padmanabhan2002b}. Such models have been extensively studied in the context of dark energy and the late-time acceleration of the universe, where the singular behavior of the Lagrangian at the boundary precisely encodes the stiff equation of state that prevents the field from violating the speed limit \cite{Pasqua2012}. Our results therefore provide a rigorous mathematical foundation for the variational analysis of these physical models: the monotonic approximation method confirms that the field never saturates the causal bound, ensuring the stability and regularity of the corresponding cosmological solutions. This connection highlights the broader significance of the framework developed in this paper, bridging nonlinear electrodynamics, cosmology, and the calculus of variations for singular functionals.

\end{document}